\documentclass[]{amsart}
\usepackage{amssymb}
\usepackage[all]{xy}

\newtheorem{theorem}{Theorem}[section]
\newtheorem{lemma}[theorem]{Lemma}
\newtheorem{corollary}[theorem]{Corollary}

\newtheorem{conjecture}[theorem]{Conjecture}

\newcommand{\Hom}{\textup{Hom}}

\title{Cluster automorphisms and compatibility of cluster variables}
\author{Ibrahim Assem, Ralf Schiffler and Vasilisa Shramchenko}
\dedicatory{Dedicated to the memory of Andrei Zelevinsky} 
\begin{document}
\begin{abstract}
In this paper, we introduce a notion of unistructural cluster algebras, for which the set of cluster variables uniquely determines the clusters. We prove that cluster algebras of Dynkin type and cluster algebras of rank 2 are unistructural, then prove that if $\mathcal{A}$ is unistructural or of Euclidean type, then
$f: \mathcal{A}\to \mathcal{A}$ is a cluster automorphism if and only if $f$ is an automorphism of the ambient field which restricts to a permutation of the  cluster variables. 
In order to prove this result, we also investigate the Fomin-Zelevinsky conjecture that two cluster variables are compatible if and only if one does not appear in the denominator of the Laurent expansions of the other.
\end{abstract}
\maketitle

\section{Introduction}\label{intro}

Cluster algebras have been introduced in 2002 by Fomin and Zelevinsky \cite{FZ1} and since then have been proved to be related to various areas of mathematics like Combinatorics, Representation Theory of Algebras, Mathematical Physics, Teichm\"uller Theory and many others. 

In \cite{ASS1} the authors have introduced the notion of cluster automorphism. Let $\mathcal A$ be a cluster algebra. A cluster automorphism of $\mathcal A$ is a $\mathbb Z$-automorphism of the algebra $\mathcal A$ mapping a cluster to a cluster and commuting with mutations. In \cite{S}, I. Saleh defines another notion of automorphism of a cluster algebra: this is an automorphism of the ambient field which restricts to a permutation of the set of cluster variables. 

In the present paper, we investigate the following conjecture. 

\begin{conjecture}
\label{conjecture1}
Let $\mathcal A$ be a cluster algebra. Then $f: \mathcal A \to \mathcal A$ is a cluster automorphism if and only if $f$ is an automorphism of the ambient field which restricts to a permutation of the set of cluster variables. 
\end{conjecture}
One implication of this conjecture follows from the fact that cluster automorphisms map each cluster to a cluster \cite[Corollary 2.7]{ASS}, so in particular cluster automorphisms permute the cluster variables.

The other implication led us to consider the following more general question. We say that a cluster algebra is {\it unistructural} if the set of cluster variables determines the cluster algebra structure, that is, there exists a unique decomposition of the set of cluster variables into clusters. We show that if a cluster algebra is unistructural, then Conjecture~\ref{conjecture1} holds true for this cluster algebra. The following conjecture is natural.

\begin{conjecture}
\label{conjecture2}
Any cluster algebra is unistructural. 
\end{conjecture}

In the present paper, we prove that cluster algebras of Dynkin type are unistructural and so are cluster algebras of rank $2$. 

Our first result asserts the validity of conjecture \ref{conjecture1} in several cases: 

\begin{theorem}
\label{theorem1}
Let $\mathcal{A}$ be a cluster algebra of Dynkin or Euclidean type or of rank $2$ or unistructural. Then $f: \mathcal{A}\to \mathcal{A}$ is a cluster automorphism if and only if $f$ is an automorphism of the ambient field which restricts to a permutation of the set of cluster variables. 
\end{theorem}

In the course of the proof, we have found that these conjectures are related to Conjecture $7.4 (2)$ \cite{FZ4} of Fomin-Zelevinsky which we restate as follows. Recall that cluster variables are called {\it compatible} if there exists a cluster containing both of them. 

\begin{conjecture}
\label{conjecture3}
Let $\mathcal{A}$ be a cluster algebra and $x, x'$ be two cluster variables. Then $x$ and $x'$ are compatible if and only if, for any cluster $\mathbf{x}$ containing $x$, the expansion ${\mathcal L}(x', \mathbf{x})$ of $x'$ as Laurent polynomial in $\mathbf{x}$ (in reduced form) is of the form $\frac{P({\mathbf x})}{m({\mathbf x}\setminus \{x\})}$,
where $P$ is a polynomial in the variables of ${\mathbf x}$ and $m$ is a monomial in the variables of ${\mathbf x}$ excluding $x$.
\end{conjecture}
In the sequel, we shall simply say that $\mathcal{L}(x',\mathbf{x})$ \emph{has no $x$ in the denominator} to express that $\mathcal{L}(x',\mathbf{x})$ is of the previous form.

This conjecture \ref{conjecture3} is proved for the case of cluster algebras arising from surfaces, see \cite[Theorem 8.6]{FST}.

We prove the following theorem. 

\begin{theorem}
\label{theorem2}
Let $\mathcal{A}_Q$ be an acyclic cluster algebra and $x, x'$ two cluster variables. 
	\begin{itemize}
		\item[(a)] If $x$ is transjective, then $x$ and $x'$ are compatible if and only if for any cluster $\mathbf x$ containing $x$, the Laurent expansion ${\mathcal L}(x', \mathbf{x})$ has no $x$ in its denominator. 
		\item [(b)] If $x$ is regular and $Q$ is Euclidean, then $x$ and $x'$ are compatible if and only if, for any cluster $\mathbf{x}'$ containing $x'$, the Laurent expansion ${\mathcal L}(x, \mathbf{x}')$ has no $x'$ in its denominator. 
	\end{itemize}
\end{theorem}

Theorem \ref{theorem2} is applied to prove Theorem \ref{theorem1}. Observe that the statement of Theorem \ref{theorem2} is not symmetric in $x$ and $x'$. This led us to the following conjecture which is weaker than Conjecture \ref{conjecture3}.

\begin{conjecture}
\label{conjecture4}
Let $\mathcal{A}$ be a cluster algebra and $x, x'$ be two cluster variables. Assume that for any cluster $\mathbf x$ containing $x$, the Laurent expansion ${\mathcal L}(x', \mathbf{x})$ has no $x$ in its denominator. Then for any cluster $\mathbf{x}'$ containing $x'$, the Laurent expansion ${\mathcal L}(x, \mathbf{x}')$ has no $x'$ in its denominator. 
\end{conjecture}

The paper is organized as follows. In Section \ref{sect 1}, we recall some preliminary notions and facts that will be useful for the proof of our results. Section \ref{sect 2} is devoted to part $(a)$ of Theorem \ref{theorem2}, while part $(b)$ is proven in Section \ref{sect 3}. Finally, we discuss unistructurality and automorphisms, then we prove Theorem \ref{theorem1} in Section \ref{sect 5}.

\section{Preliminaries}\label{sect 1}
\subsection{Cluster algebras.} Let $Q$ be a finite connected quiver without loops or two-cycles. We denote by $n=|Q_0|$ the number of points of $Q$, the points are denoted by $1,2,\dots, n$. Let $\mathbf x = \{x_1, \dots, x_n\}$ be a set of $n$ variables, called a cluster, where we agree that the variable $x_i$ corresponds to the point $i$ of the quiver. The pair $(\mathbf x, Q)$ is called a {\it seed}. The field $\mathcal F = \mathbb Q(x_1,\dots, x_n)$  of rational functions in $x_1,\dots, n_n$ is called the {\it ambient field}.

Let $k$ be such that $1\leq k\leq n$. The seed mutation $\mu_k$ in direction $k$ transforms $(\mathbf{x}, Q)$ into $\mu_k(\mathbf{x}, Q) = (\mathbf{x}', Q')$ defined as follows: 

First, $Q'$ is obtained from $Q$ by applying the following operations
\begin{itemize}
			\item for any path $i\to k \to j$,  a new  arrow $i\to j$ is  inserted; 
			\item all arrows incident to $k$ are reversed;
			\item all ensuing two-cycles are deleted. 
			\end{itemize}		
			
Next, $\mathbf x' = (\mathbf x\setminus \{x_k\})\bigcup \{x_k'\}$ where $x_k'\in\mathcal F$ is defined by the so-called {\it exchange relation}
			\begin{equation*}
			x_k' = \frac{\underset{i\to k}{\prod} x_i + \underset{k\to i}{\prod}x_i}{x_k},
			\end{equation*}
			where products are taken over arrows entering and leaving $k$, respectively.

Iterating this procedure, we obtain a set $\{({\mathbf x}_\alpha, Q_\alpha)\}_\alpha$ of seeds, where the $\mathbf{x}_\alpha$ are the clusters and the $Q_\alpha$ are the exchange quivers. The cluster algebra $\mathcal A = \mathcal A(\mathbf x, Q)$, with initial seed $(\mathbf x, Q)$ is the $\mathbb{Z}$-subalgebra of $\mathcal F$ generated by the union $\mathcal X = \bigcup_\alpha \mathbf x_\alpha$ of all possible clusters obtained from $\mathbf x$ by successive mutations. Elements of $\mathcal X$ are called cluster variables. Two cluster variables are called \emph{compatible} if there exists a cluster containing both.

One of the most remarkable results of the theory is the Laurent phenomenon \cite{FZ1} which asserts that for any cluster algebra $\mathcal A$ and any seed $(\mathbf x, Q)$ of $\mathcal A$, each cluster variable $x$ of $\mathcal A$ is a Laurent polynomial over $\mathbb {Z}$ in the cluster variables from $\mathbf x = (x_1, \dots, x_n)$, that is $x$ can be written as 
\begin{equation*}
x = \mathcal L(x, \mathbf x) = \frac{p(\mathbf x)}{x_1^{d_1}, \dots, x_n^{d_n}}
\end{equation*}
as a reduced fraction, that is, $p$ is not divisible by any of the $x_i$. The vector $\mathbf d = (d_1, \dots, d_n)$ is called the {\it denominator vector} of $x$ and denoted by $\mathbf {den}(x)$.

A famous conjecture is the so-called {\it positivity conjecture}. It says that, for any cluster algebra $\mathcal A$, any cluster variable $x$ and any seed $(\mathbf x, Q)$ of $\mathcal A$, the numerator $p(\mathbf x)$ of $ \mathcal L(x, \mathbf x)$ has coefficients which are nonnegative integers.
The positivity conjecture has been proved in  several important cases, most notably for cluster algebras arising from surfaces \cite{MSW} and for acyclic cluster algebras \cite{KQ}, that is, cluster algebras arising from a quiver which can be transformed to an acyclic quiver using a sequence of mutations, and finally in \cite{LS4} for all cluster algebras arising from quivers.

\subsection{The cluster category.} Let $Q$ be an acyclic quiver, and $k$ be an algebraically closed field. We denote by $kQ$ the path algebra of $Q$, by ${\rm mod} \,kQ$ the category of finitely generated $kQ$-modules and by $\mathcal D = \mathcal D^b({\rm mod}\,kQ)$ the bounded derived category over ${\rm mod} \,kQ$. The cluster category $\mathcal C = \mathcal C_Q$ is the orbit category of $\mathcal D$ under the action of the automorphism $\tau^{-1}[1]$, where $\tau$ is the Auslander-Reiten translation and $[1]$ is the shift in $\mathcal D$, see \cite{BMRRT}. Then $\mathcal C$ is a triangulated 2-Calabi-Yau category having almost split triangles. The Auslander-Reiten quiver $\Gamma(\mathcal C)$ of $\mathcal C$ is the quotient of $\Gamma(\mathcal D)$ under the action of the quiver automorphism induced by the functor $\tau^{-1}[1]$. Note that the indecomposable objects of $\mathcal C$ may be identified with indecomposable $kQ$-modules or with indecomposable summands of $kQ[1]
  = \{ P_{i}[1]\,\vert\, i\in Q_0\}$, the shifts of the indecomposable projective $kQ$-modules. 
Then $\Gamma({\mathcal C})$ always has a unique component containing all objects of $kQ[1]$. This component is called {\it transjective}. If $Q$ is Dynkin, then $\Gamma({\mathcal C})$ reduces to the transjective component (and is finite). Otherwise the transjective component is of the form $\mathbb{Z}Q$ (see \cite{ASS}, VIII 1.1) and there are infinitely many other components called {\it regular}, which are either stable tubes (if $Q$ is Euclidean) or of type $\mathbb {ZA}_\infty$ (if $Q$ is wild). 

Let $n=|Q_0|$. A (basic) \emph{tilting object} $T$ in $\mathcal C_Q$ is an object of the form $T=\oplus_{i=1}^n T_i$ where the $T_i$ are indecomposable nonisomorphic objects of $\mathcal C$ such that ${\rm Hom}_{\mathcal C} (T_i, T_j[1])=0$ for all $i,j$. In particular, we have ${\rm Hom}_{\mathcal C} (T_i, T_i[1])=0$ for all $i$, that is, each $T_i$ is a {\it rigid} object of $\mathcal C$. With each tilting object $T$ is associated a map
\begin{equation*}
X_?^T:\mathcal C_0 \to \mathbb Z[x_1^{\pm 1}, \dots, x_n^{\pm 1}]
\end{equation*}
called the {\it cluster character}. This map induces a bijection between the indecomposable rigid objects of $\mathcal C$ and the cluster variables in the cluster algebra $\mathcal A(\mathbf x, Q)$. This bijection also induces a bijection between the tilting objects in $\mathcal C$ and the clusters of $\mathcal A(\mathbf x, Q)$, see \cite{CK2, Palu}

An algebra $\mathcal B$ is called {\it cluster-tilted of type $Q$} if there exists a tilting object $T$ in $\mathcal C_Q$ such that $\mathcal B = {\rm End}_{\mathcal C_Q}T$. Let $({\rm add}(\tau T))$ be the ideal of $\mathcal C = \mathcal C_Q$ consisting of those morphisms factoring through objects of ${\rm add}(\tau T)$. Then the functor ${\rm Hom}_\mathcal C (T, -)$ induces an equivalence $\mathcal C/({\rm add}(\tau T))\simeq {\rm mod}\,B$ (see \cite{BMR1}). For any $i\in Q_0$, we denote by $S_i$ the simple $B$-module corresponding to $i$. We denote by $\langle-, - \rangle $ the bilinear form on ${\rm mod} \,B$ defined by 
\begin{equation*}
\langle M,N \rangle   = {\rm dim\, Hom}_B(M,N) - {\rm dim\, Ext}_B^1(M,N)
\end{equation*}
for any $B$-modules $M$ and $N$. Further, we write 
\begin{equation*}
\langle M,N \rangle _a  =\langle M,N \rangle  - \langle N,M \rangle ,
\end{equation*}
 that is, $\langle -,- \rangle _a$ is the antisymmetrised form of $\langle -,- \rangle $.
Let $\mathbf e\in\mathbb{N}^{Q_0}$ and ${\rm Gr}_\mathbf e M$ denote the set of submodules of $M$ having dimension-vector equal to $\mathbf e$. This set is a projective variety, called the {\it Grassmannian} of submodules of $M$ of dimension vector $\mathbf{e}$. We denote by $\chi({\rm Gr}_\mathbf e M)$ its Euler-Poincar\'e characteristic (with respect to the singular cohomology if $k$ is the field of complex numbers, and to the \'etale cohomology if $k$ is arbitrary).  Then the cluster character $X_?^T$ is the unique map such that:
\begin{itemize}
\item [a)] $X^T_{T_i[1]}=x_i$ for any $i\in Q_0$;
\item [b)] if $M$ is indecomposable and not isomorphic to any $T_i[1]$ then $$X_M^T = \sum _{e\in\mathbb N^{Q_0}} \chi({\rm Gr}_e({\rm Hom_{\mathcal{C}}(T,M)})) \prod_{i\in Q_0} x_i^{\langle S_i, e \rangle _a - \langle S_i, {\rm Hom}_{\mathcal{C}}(T, M) \rangle };$$
\item [c)] for any two objects $M$, $N$ in $\mathcal C$
	\begin{equation*}
	\chi^T_{M\oplus N} = \chi^T_M\chi^T_N. 
	\end{equation*}
\end{itemize}

A cluster variable $x$ in $\mathcal A(\mathbf x, Q)$ is called {\it transjective} (or {\it regular}) if $x=X^T_M$ with $M$ an indecomposable rigid transjective (or regular, respectively) object of $\mathcal C$. 

It is shown in \cite{BMRdenominators} that if $T$ has only transjective direct summands and $x=X^T_M$ is any cluster variable, then 
\begin{equation*}
\mathbf{{\rm \mathbf {den}}}(x) = {{\rm \mathbf {dim}}}\,{\rm Hom}_{\mathcal C} (T,M). 
\end{equation*}
Moreover, if $Q$ is Euclidean, and no indecomposable regular direct summand of $T$ is of quasilength $r-1$ lying in a tube of rank $r$, then this equality holds for all cluster variables.

\section{The transjective case}\label{sect 2}
\subsection{The motivating example and Conjecture \ref{conjecture3}} In \cite{AD}, it was asked when indecomposable rigid modules over cluster-tilted algebras are uniquely determined by their composition factors. This is related to problem 7.6 of \cite{FZ4}. The main result of \cite{AD} says that if $B$ is a cluster-tilted algebra and $M,N$ are indecomposable transjective modules (in particular, are rigid) then $M\cong N$ if and only if $\underline{\textup{dim}}\,M=\underline{\textup{dim}}\,N$. This is not true if $M, N$ are not transjective. Indeed, let $B$ be the cluster-tilted algebra of type $\tilde{\mathbb{A}}_{1,2}$ given by the quiver
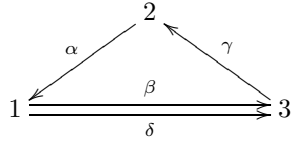
\[\xymatrix@C40pt{&2\ar[ld]_\alpha \\ 
1\ar@<2pt>[rr]^\beta\ar@<-2pt>[rr]_\delta&&3\ar[lu]_\gamma}\]
bound by $\alpha\beta=\beta\gamma=\gamma\alpha=0$. Then $\Gamma(\textup{mod}\,B)$ contains a tube of rank 2 containing the projective-injective module $P_2=I_2$ corresponding to the point 2 of the quiver:
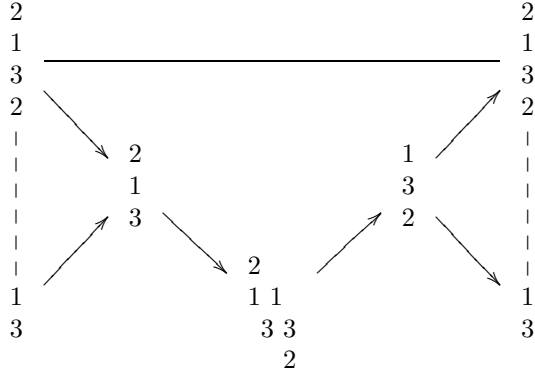
\[ \xymatrix@R=0pt{
{\begin{array}{c}2\\1\\3\\2\end{array}} \ar[rd] \ar@{-}[rrrr]\ar@{--}[dd]&&&&
{\begin{array}{c}2\\1\\3\\2\end{array}}\ar@{--}[dd]
\\
&{\begin{array}{c}2\\1\\3\end{array}} \ar[rd] &&
{\begin{array}{c}1\\3\\2\end{array}}\ar[ru]\ar[rd]
\\
{\begin{array}{c}1\\3\end{array}}\ar[ru]
&&
{\begin{array}{l}2\ \ \ \\ 1\ 1\\\ \, 3\ 3\\\ \ \ \ 2\end{array}}\ar[ru]
&&
{\begin{array}{c}1\\3\end{array}}
}\]
where we identify along the dotted lines and ${\begin{array}{c}2\\1\\3\\2\end{array}} $ lies on the mouth.
Then ${\begin{array}{c}2\\1\\3\end{array}}$ and ${\begin{array}{c}1\\3\\2\end{array}}$ are non-isomorphic indecomposable rigid $B$-modules with the same dimension vector. Let $T_2$ be the object in the cluster category corresponding to the projective $B$-module $P_2$, then the corresponding cluster variable is 
\[ \frac{2x_1x_3+x_3^2+x_2+x_1^2}{x_1x_2x_3}.
\]
In particular, its denominator vector is $(1,1,1)$ while its dimension vector is $(1,2,1)$.
Note however that in this example the $i$-th component of the denominator vector is nonzero and also the $i$-th component of the dimension vector is nonzero.
This led us to consider conjecture 7.4(2) of \cite{FZ4} which we  have restated as Conjecture~\ref{conjecture3}, see  the introduction.

%
%\begin{conjecture}
% Let $\mathcal{A}$ be a cluster algebra and $x,x'$ be two cluster variables. Then $x$ and $x'$ are compatible if and only if, for any cluster $\mathbf{x}$ containing $x$, the expansion $\mathcal{L}(x',\mathbf{x})$ of $x'$ as a Laurent polynomial in $\mathbf{x}$ is of the form 
% \[\mathcal{L}(x',\mathbf{x}) =\frac{p(\mathbf{x})}{m(\mathbf{x}\setminus\{x\})} ,
% \] 
%where $p$ is a polynomial in $\mathbf{x}$ and $m$ is a monomial in $\mathbf{x}\setminus\{x\}$.
%\end{conjecture}
%
%In the sequel, we shall simply say that $\mathcal{L}(x',\mathbf{x})$ \emph{has no $x$ in the denominator} to express that $\mathcal{L}(x',\mathbf{x})$ is of the previous form. 

One implication of the conjecture is easy.

\begin{lemma}\label{lem 2.1}
 Let $\mathcal{A}$ be a cluster algebra and $x,x'$ be two cluster variables. If $x,x'$ are compatible then, for any cluster $\mathbf{x}$ containing $x$,  $\mathcal{L}(x',\mathbf{x})$ has no $x$ in the denominator.
\end{lemma}
\begin{proof}
 Let $\mathbf{x}$ be a cluster containing $x$. Because $x$ and $x'$ are compatible, there exists a sequence of mutations (not involving $\mu_x$) such that $\mathbf{x}$ is mutated to a new cluster $\mathbf{x}'$ containing $x$ and $x'$ simultaneously. Expanding $x'$ in terms of the original cluster $\mathbf{x}$, we get  $\mathcal{L}(x',\mathbf{x})$ of the required form.
 \end{proof}

\subsection{Proof of  Theorem \ref{theorem2} (a)}
We start with an easy technical lemma.  Throughout, we assume that $\mathcal{A}=\mathcal{A}(\mathbf{x},Q)$ is an acyclic cluster algebra and denote by $\mathcal{C}=\mathcal{C}_Q$  the associated cluster category.

\begin{lemma}\label{lem 2.2}
 Let $x,x'$ be two cluster variables in $\mathcal{A}$ and $T$ be a tilting object in $\mathcal{C}$ such that $X^T_{T_i[1]}=x$. If $M'\in \mathcal{C}$ is such that $X_{M'}^T=x'$, then $x$ and $x' $ are compatible  if and only if $\Hom_{\mathcal{C}}(T_i,M')=0$.
\end{lemma}
\begin{proof}
 Indeed, $x$ and $x'$ are compatible if and only if the corresponding objects $T_i[1]$ and $M'$ are compatible, that is, $\Hom_{\mathcal{C}}(T_i[1],M'[1])=0$. But this is equivalent to $\Hom_{\mathcal{C}}(T_i,M')=0$.
\end{proof}

We are now able to state and prove the following  theorem.

\begin{theorem}\label{thm 2.2}
 Let $\mathcal{A}$ be a cluster algebra, and $x,x'$ be two cluster variables with $x$ transjective. Then $x$ and $x'$ are compatible if and only if, for any cluster $\mathbf{x}$ containing $x$,  the Laurent expansion $\mathcal{L}(x',\mathbf{x})$ has no $x$ in the denominator.
\end{theorem}
 
\begin{proof}
 Because of Lemma \ref{lem 2.1}, we only need to prove the sufficiency. Assume to the contrary that $x,x'$ are not compatible and let $\mathbf{x}$ be a transjective cluster containing $x$. Such a cluster always exists, take for instance a slice in the cluster category containing the object corresponding to $x$, see \cite{ABS2}. Let $T\in \mathcal{C}$ be the tilting object corresponding to $\mathbf{x}$. Then $T$ is transjective and there exists an indecomposable summand
 $T_i$ of $T$ such that $X^T_{T_i[1]}=x$. Let also $M'$ be the object in $\mathcal{C}$ such that $X_{M'}^T=x'$. Then, because of
\cite{BMRdenominators} (see section \ref{sect 1} above), we have ${\textbf{den}}(x')={\textbf{dim}}\,\Hom_{\mathcal{C}}(T,M')$. 
On the other hand, because of  
  the previous lemma, the incompatibility of $x$ and $x'$ implies that $\Hom_{\mathcal{C}}(T_i,M')\ne 0$. This shows that 
 \[\mathcal{L}(x',\mathbf{x}) = \frac{P(\mathbf{x})}{x^d \,m(\mathbf{x}\setminus\{x\})},\]
 with $d>0$. The proof is now complete.
\end{proof}

The main application of this theorem is to cluster automorphisms (see section \ref{sect 5} below). For the time being, we obtain an obvious corollary asserting the truth of the previous conjecture in two particular cases.

\begin{corollary}\label{cor 2.2}
Let $\mathcal{A}$ be a cluster algebra of Dynkin type, or of rank 2, then two cluster variables 
  $x$ and $x'$ are compatible if and only if, for any cluster $\mathbf{x}$ containing $x$,  the Laurent expansion $\mathcal{L}(x',\mathbf{x})$ has no $x$ in the denominator.
\end{corollary}
\begin{proof}
 Indeed, in these cases, all cluster variables are transjective.
\end{proof}

\subsection{Remark}
 The statement of the theorem above is clearly not symmetric:  it says that if one of the variables is transjective, then compatibility means that the Laurent expansion of the other does not have the first in the denominator.

\section{The tame case}\label{sect 3}
\subsection{Preliminaries on the cluster category}\label{sect 3.1}
As in section \ref{sect 2}, we let $\mathcal{A}$ be a cluster algebra of type $Q$, and we denote by $\mathcal{C}$ the associated cluster category. Throughout this section, we assume that $Q$ is a Euclidean quiver.

Assume that $T_1$ is an indecomposable rigid object of quasilength $r-1$ lying in a tube $\mathcal{T}$ of $\Gamma(\mathcal{C})$ of rank $r$. We denote by 
$\xymatrix{T_1\ar[r]&T_2\ar[r]&\cdots\ar[r]&T_{r-1}}$
the sectional path from $T_1$ to the mouth of $\mathcal{T}$ and assume that $T$ is a tilting object having all the $T_i$ (with $1\le i\le r-1$) as summands. Such an object exists, because the $T_i$ are clearly compatible in $\mathcal{C}$.

Consider the set $\Delta$ of all indecomposables in $\mathcal{T}$ lying on a path of irreducible morphisms from $T_1$ to $\tau^2 T_1$ of length equal to $2r-4$. Observe that, for any $M\in \Delta$, we have $\Hom_{\mathcal{C}}(M,\tau^2 T_1)\ne 0$. Actually, if $M$ is a rigid object in $\mathcal{T}$ such that $\Hom_{\mathcal{C}}(T_1,M)\ne 0$, then $M$ belongs to $\Delta$ (and even to the ``upper half'' of $\Delta$).
Note that in \cite{BM}, the authors define the notion of a wing $\mathcal{W}_{\tau T_1}$ attached to $T_1$. This is related to $\Delta$ as follows: if $X$ is a rigid object in $\mathcal{T}$, then $X\in \Delta$ if and only if $X\notin  \mathcal{W}_{\tau T_1}$.

\begin{lemma}
 \label{lem 3.1}
 $\Delta$ contains no direct summand of $\tau T$.
\end{lemma}
\begin{proof}
 Indeed, if $\tau T_i\in \Delta$ for some $i$ such that $1\le i <r$, then there is a morphism from $T_1 $ to $\tau T_i$ and thus $\textup{Ext}^1_{\mathcal{C}}(T_1,T_i)\ne 0$, a contradiction.
\end{proof}

Let $U_1$ denote the unique direct predecessor of $T_1$ of quasilength $r$. Then we have the following picture.
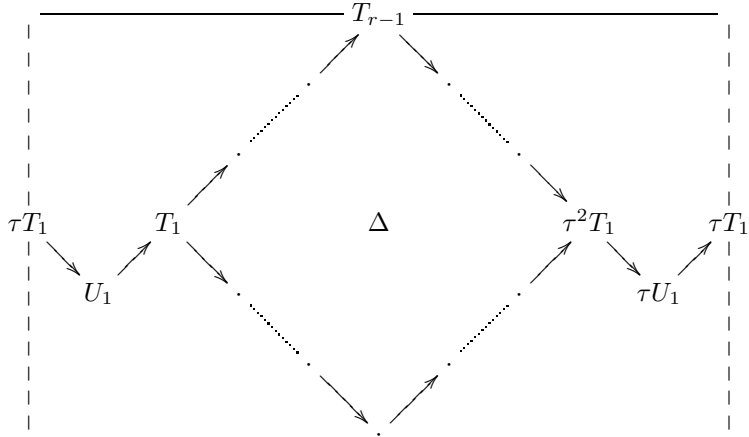
\[\xymatrix@!@R0pt@C0pt{~ \ar@{-}[rrrrr]  \ar@{--}[ddd]&&&&&T_{r-1}\ar[rd] \ar@{-}[rrrrr] &&&&&~  \ar@{--}[ddd]\\
&&&&\cdot\ar[ru]&&\cdot\ar@{.}[rd] \\
&&&\cdot\ar@{.}[ru]&&&&\cdot\ar[rd] \\
\tau T_1\ar[rd] \ar@{--}[ddd]&&T_1\ar[ru]\ar[rd] &&&\Delta&&&\tau^2T_1\ar[rd]&&\tau T_1 \ar@{--}[ddd]\\
&U_1\ar[ru]&&\cdot\ar@{.}[rd] &&&&\cdot\ar[ru] &&\tau U_1\ar[ru]\\
&&&&\cdot\ar[rd] &&\cdot\ar@{.}[ru]\\
&&&&&\cdot\ar[ru]&&&&&
}\]
where we identify along the vertical dashed lines and $T_{r-1}$ lies on the mouth of $\mathcal{T}$.

\begin{lemma}
 \label{lem 3.1.2}
 With the notation above, we have
 \begin{itemize}
\item [\textup{(a)}] $\textup {dim}\,\Hom_{\mathcal{C}}(T_1,M)=2$ for every $M\in \Delta$.
\item [\textup{(b)}]  $\textup {dim}\,\Hom_{\mathcal{C}}(U_1,M)=2$ if $M\in \Delta$ does not lie on the sectional path from the mouth to $\tau U_1$.
\item [\textup{(c)}]  $\textup {dim}\,\Hom_{\mathcal{C}}(U_1,M)=1$ if $M\in \Delta$  lies on the sectional path from the mouth to $\tau U_1$.
\end{itemize}
\end{lemma}

\begin{proof}
 Recall that, for all $X,Y\in \mathcal{T}$,
 \[\begin{array}{rcl}\Hom_{\mathcal{C}}(X,Y) &=& 
 \Hom_{\mathcal{D}}(X,Y) \oplus  \Hom_{\mathcal{D}}(\tau X [-1], Y) \\
 &=& 
 \Hom_{kQ}(X,Y) \oplus  \textup{Ext}^1_{kQ}(\tau X , Y) .
 \end{array}\]
 Now, because $\mathcal{T}$ is standard in $\textup{mod}\,kQ$, for $M\in \Delta$, we have
 \[ \textup{dim}\,\Hom_{kQ}(T_1,M)=1,\]
 while $ \textup{dim}\,\textup{Ext}^1_{kQ}(\tau T_1 ,M) = \textup{dim}\,\Hom_{kQ}(M,\tau^2 T_1)=1$, where we have applied the Auslander-Reiten formula.
 Similarly, for $M\in \Delta$, we have
 \[ \textup{dim}\,\Hom_{kQ}(U_1,M)=1,\]
 while $ \textup{dim}\,\textup{Ext}^1_{kQ}(\tau U_1 ,M) = \textup{dim}\,\Hom_{kQ}(M,\tau^2 U_1)$ is equal to 1 or 0 according to the cases considered in (b) or (c) respectively. Here, we use essentially that $T_1$ and $U_1$ are of quasi-length $r-1$ and $r$, respectively. 
\end{proof}

\subsection {Passing to cluster-tilted algebras}
Let now $B=\textup{End}_{\mathcal{C}}\,T$ be the cluster-tilted algebra corresponding to the tilting object $T$ in $\mathcal{C}$ considered in subsection \ref{sect 3.1}. Then $P_1=\Hom_{\mathcal{C}}(T,T_1) $ is projective, and there is a sectional path of projective modules of length $r-1$ from $P_1$ to the mouth of the tube. On the other hand, $I_1=\Hom_{\mathcal{C}}(T,\tau^2 T_1)$ is the injective module corresponding to $P_1$ (see Lemma 5 of \cite{ABS4}).

Let $\Omega=\{\Hom_{\mathcal{C}}(T,M) \mid M\in \Delta\}$. Observe that, because of Lemma \ref{lem 3.1}, for every $M\in \Delta$, the $B$-module $\Hom_{\mathcal{C}}(T,M)$ is indecomposable. Let $R_1=\Hom_{\mathcal{C}}(T,U_1)$. Then $R_1$ is equal to the radical of the indecomposable projective $B$-module $P_1$, because the unique direct successor $P_2=\Hom_{\mathcal{C}}(T,T_2)$ of $P_1$ of smaller quasi-length  is projective.

\begin{lemma}
 \label{lem 3.2}
 With the above notation, for every $N\in \Omega$, we have
 \begin{itemize}
\item [\textup{(a)}] 
 $ \textup{dim}\,\Hom_{B}(P_1,N)=2.$
\item [\textup{(b)}] 
 $ \textup{dim}\,\Hom_{B}(R_1,N)=1.$
\end{itemize}
\end{lemma}

\begin{proof}
 We recall from section \ref{sect 1} that $\textup{mod}\, B\cong \mathcal{C}/ \textup{add}\tau T$. Let $M\in \mathcal{C}$ be such that $\Hom_{\mathcal{C}}(T,M)=N$. We first observe that, because $T$ is a tilting object in $\mathcal{C}$, then no morphism from $T_1$ to $M$ factors through $\textup{add} \tau T$. Therefore Lemma \ref{lem 3.1} implies (a).
 
 In order to prove (b), we recall that, according to  Lemma \ref{lem 3.1}, a basis of the vector space $\Hom_{\mathcal{C}}(U_1,M)$ consists of at most two morphisms, one of them constituting a basis of $\Hom_{kQ}(U_1,M)$ and the other (if nonzero) constituting a basis of $\textup{Ext}^1_{kQ}(\tau U_1, M)$. Now, note that, if $g$ is a basis vector in $\Hom_{kQ}(U_1,M) $ then it clearly does not factor through $\textup{add}\,\tau T$ and so $\Hom_{\mathcal{C}}(T,g):R_1\to N$ is nonzero. Thus, in order to prove (b), it suffices to show that, if $\textup{Ext}^1_{kQ}(\tau U_1, M)\ne 0$, then a basis vector $\xi$ of this vector space factors through $\tau T$.
 
 We construct a morphism $f$ from $\tau T$ to $M$. Indeed the sectional path in $\mathcal{T}$ from the mouth to $M$ and the sectional path from $\tau T_1$ to the mouth intersect in an indecomposable summand $\tau T_i$ of $\tau T$. Let $f$ denote the composition of the morphisms lying on the sectional path from $\tau T_i$ to $M$. Then $f\ne 0$ (see \cite[Corollary IX 2.2]{ASS}).
 
 Now, observe that we have a nonsplit short exact sequence in $\textup{mod}\,kQ$
 \[\zeta: \xymatrix{0\ar[r]&\tau T_i\ar[r] & E\ar[r] & \tau U_1\ar[r] & 0},\]
 with $E$ indecomposable.
 Because $\textup{Ext}^1_{kQ}(\tau U_1, M)$ is one dimensional, there exists a scalar $\lambda\in k$ such that $\xi=\zeta(\lambda f)$, that is, we have a commutative diagram with exact rows
 \[ \xymatrix{\zeta: &0\ar[r]&\tau T_i\ar[r]\ar[d]_{\lambda f}& E\ar[r] \ar[d]& \tau U_1\ar[r]\ar[d]^= & 0\\
\xi:& 0\ar[r]& M\ar[r]&E\oplus F\ar[r]&\tau U_1\ar[r]&0.}
 \] 
 Indeed, the middle term of $\xi$ is of this form, as seen in the picture below. 
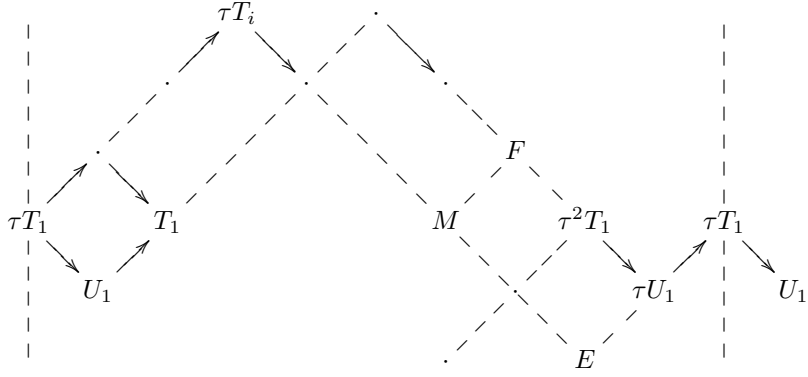
\[\xymatrix@!@R0pt@C0pt{  \ar@{--}[ddd]&&&\tau T_i\ar[rd]&&\cdot\ar[rd]  &&&&&~  \ar@{--}[ddd]\\
&&\cdot\ar[ru]&&\cdot\ar@{--}[rrdd]\ar@{--}[ru]&&\cdot\ar@{--}[rd] \\
&\cdot\ar@{--}[ru]\ar[rd]&&&&&&F\ar@{--}[rd] \\
\tau T_1\ar[ru]\ar[rd] \ar@{--}[dd]&&T_1\ar@{--}[rruu] &&&&M\ar@{--}[ru]\ar@{--}[rrdd]&&\tau^2T_1\ar[rd]&&\tau T_1\ar[dr] \ar@{--}[dd]\\
&U_1\ar[ru]&& &&&&\cdot\ar@{--}[ru] &&\tau U_1\ar[ru]&&U_1\\
&&&& &&\cdot\ar@{--}[ru]&&E\ar@{--}[ru]&&\\
}\]
The fact that $\xi=\zeta(\lambda f)$ shows that $\xi$ factors through $\tau T $ in $\mathcal{C}$. This completes the proof.
\end{proof}

\begin{corollary}
 \label{cor 3.2} 
 With the above notation, letting $S_1=\textup{top}\,P_1$, we have $\langle S_1, N\rangle = 1$, for all $N\in \Omega$.
\end{corollary}

\begin{proof}
 Recall that $\langle S_1, N\rangle = \textup{dim}\,\Hom_B(S_1,N)-\textup{dim\,Ext}^1_B(S_1,N).$ Now, applying the functor $\Hom_B(-,N)$ to the short exact sequence 
\[\xymatrix{0\ar[r]& R_1\ar[r]&P_1\ar[r]&S_1\ar[r]&0}\] in $\textup{mod}\, B$ yields an exact sequence
 \[ \xymatrix@C15pt{0\ar[r]&\Hom_B(S_1,N)\ar[r]&\Hom_B(P_1,N)\ar[r]&\Hom_B(R_1,N)\ar[r]&\textup{Ext}^1_B(S_1,N)\ar[r]&0}
 \]
 from which we deduce $\langle S_1, N\rangle =\textup{dim}\,\Hom_B(P_1,N)-\textup{dim\,Hom}_B(R_1,N)$, and the latter equals 1 because of the lemma.
\end{proof}

\subsection{Proof of Theorem \ref{theorem2} (b)} 
We are now in a position to prove the second part of our Theorem \ref{theorem2}.

\begin{theorem}
 \label{thm 3.3}
 Let $\mathcal{A}$ be a cluster algebra of Euclidean type and $x,x'$ be two cluster variables with $x'$ regular. Then $x$ and $x'$ are compatible if and only if, for any cluster $\mathbf{x}$ containing $x$, the Laurent expansion $\mathcal{L}(x',\mathbf{x})$ has no $x$ in its denominator.
\end{theorem}

\begin{proof}
 Because of Lemma \ref{lem 2.1}, we only need to prove the sufficiency. Assume that $x,x'$ are not compatible, with $x'$ regular. Because of Theorem \ref{thm 2.2}, we may assume that $x$ is regular as well. Since there are no extensions between different tubes, we may assume that they correspond to objects in the cluster category lying in the same tube $\mathcal{T}$. We denote by $r$ the rank of $\mathcal{T}$. Let $T_1$ and $M$ be such that $X^T_{T_1[1]}=x$ and $X^T_M=x'$. If the quasilength of $T_1$ is at most $r-2$ then, because of \cite{BMRdenominators}, we have ${\textbf{den}} \, x' ={\textbf{dim}}\,\Hom_{\mathcal{C}}(T,M)$. Applying the same argument as in the proof of Theorem \ref{thm 2.2} yields the conclusion.  
 
 Assume thus that the quasilength of $T_1$ is $r-1$. We may assume that the objects lying on the sectional path from $T_1$ to the mouth are summands of the tilting object $T$. Lemma \ref{lem 2.2} implies that $\Hom_{\mathcal{C}}(T_1,M)\ne 0$, and since $M$ is an exceptional object in $\mathcal{T}$, we have  $M\in \Delta$. We now recall that
 \[
 X_M^T= \sum_{\mathbf{e}\in \mathbb{N}^{\mathbb{Q}_0}}
\chi(\textup{Gr}_{\mathbf{e}}(\Hom_{\mathcal{C}}(T,M))
\prod_{i\in Q_0} x_i^ {\langle S_i,\mathbf{e}\rangle_a-\langle S_i, \Hom_{\mathcal{C}}(T,M)\rangle} .\]
The summand corresponding to the vector $\mathbf{e}=0$ is
\[\prod_{i\in Q_0} x_i^ {-\langle S_i, \Hom_{\mathcal{C}}(T,M)\rangle} .\]
Because of Corollary \ref{cor 3.2}, we have 
\[\langle S_1, \Hom_{\mathcal{C}}(T,M)\rangle =1.\]
In other words, the Laurent monomial 
\[\prod_{i\in Q_0} x_i^ {-\langle S_i, \Hom_{\mathcal{C}}(T,M)\rangle} \]
contains the variable $x(=x_1)$ in its denominator. Because of the positivity theorem (see \cite{LS4}), this Laurent monomial is not cancelled by other summands. Therefore we have 
\[X^T_M=\frac{p(\mathbf{x})}{x^d\,m(\mathbf{x}\setminus\{x\})}\]
with $d>0$. This completes the proof.
\end{proof}

\section{Unistructurality and automorphisms}\label{sect 5}

\subsection{} 
We now define the notion of unistructurality.

Let $\mathcal{A}=\mathcal{A}(\mathbf{x},Q)$ be a cluster algebra with initial seed $(\mathbf{x},Q)$. This initial seed and the mutation process yield a family $(\mathbf{x}_{\alpha})_{\alpha\in\Gamma_0}$ of clusters, where $\Gamma_0$ is the set of vertices of the corresponding exchange graph $\Gamma$.
Let $\mathcal{X}=\cup_{\alpha\in\Gamma_0}\mathbf{x}_\alpha$ be the set of cluster variables.

We say that $\mathcal{A}$ is \emph{unistructural} if, for any subset $\mathbf{x}'$ of $\mathcal{X}$ and quiver $Q'$ such that the pair $(\mathbf{x}',Q')$ generates by mutation the same set $\mathcal{X}$ of cluster variables, then the exchange graphs and the set of clusters of $\mathcal{A}$ and $\mathcal{A}(\mathbf{x}',Q')$ are the same. More precisely,
if $(\mathbf{x}',Q')$ generates by mutation a family of clusters $(\mathbf{x}'_\beta)_{\beta\in\Gamma'_0}$ where $\Gamma_0'$ is the set of vertices of the corresponding exchange graph $\Gamma'$, then the equality 
$\mathcal{X}=\cup_{\beta\in\Gamma'_0}\mathbf{x}'_\beta$ of the set of cluster variables implies that $\Gamma=\Gamma'$ and there exists a permutation $\sigma$ of $\Gamma_0$ such that $\mathbf{x}_\alpha = \mathbf{x}'_{\sigma(\alpha)} $ for any $\alpha$.

Notice that under the hypothesis that $\mathcal{X}=\cup _{\beta\in\Gamma'_0}\mathbf{x}'_\beta$ the ranks of $\mathcal{A}$ and $\mathcal{A}(\mathbf{x}',Q
)$
are necessarily the same: indeed, this rank is the cardinality of a transcendence basis of the (common) ambient field.

 As stated in the introduction, we conjecture that every cluster algebra is unistructural. In this section, we prove the conjecture for cluster algebras of Dynkin type and for those  of rank two. For this proof we use Theorem \ref{theorem1} which we prove first in the two following lemmata. 

%The following two lemmas are parts of Theorem \ref{theorem1}. 

\begin{lemma}
\label{lemma_5.1}
Let $Q$ be a Dynkin quiver, or a Euclidean quiver or a quiver with two points.  Then $f: \mathcal{A}\to \mathcal{A}$ is a cluster automorphism if and only if $f$ is an automorphism of the ambient field which restricts to a permutation of the set of cluster variables. 
\end{lemma}

\begin{proof} According to the remark following Conjecture \ref{conjecture1}, it suffices to show sufficiency. Let $f$ be an automorphism of the ambient field which restricts to a permutation of the set of cluster variables. 
Because of Corollary 2.7 of \cite{ASS1}, it suffices to prove that $f$ maps every cluster of $\mathcal A$ onto another cluster. Let $\mathbf u$ be a cluster in $\mathcal A$ such that $f(\mathbf u)$ is not a cluster. Then $\mathbf u$ contains two compatible cluster variables $u$ and $v$ whose images $y=f(u)$ and $x=f(v)$ are not compatible. 

Assume first that $x$ is transjective. Then, due to Theorem \ref{thm 2.2} , for any cluster $\mathbf x$ containing $x$, the expansion $\mathcal L(y, \mathbf x)$ has $x$ in its denominator, that is 
\begin{equation*}
y = \mathcal L(y, \mathbf x) = \frac{p(\mathbf x)}{x^d\,m(\mathbf x\setminus \{x\})}
\end{equation*}
with $d>0$ and $x$ not dividing $p(\mathbf x)$. Because $f$ is a field automorphism, so is $f^{-1}$ and we have 
\begin{equation*}
u =  \frac{p(f^{-1}(\mathbf x))}{v^d\,m(f^{-1}(\mathbf x)\setminus \{v\})}
\end{equation*}
with $d>0$ and $v$ not dividing $p(f^{-1}(\mathbf x))$. But   because of Lemma \ref{lem 2.1}, this implies that $u$ and $v$ are not compatible, a contradiction. 

We are thus reduced to the case where $x$ and $y$ are both regular cluster variables. In this case, because of Theorem \ref{thm 3.3}, for any cluster $\mathbf x$ containing $x$, the expansion $\mathcal L(y, \mathbf x)$ has $x$ in its denominator and the above reasoning finishes the proof. 
\end{proof}

\begin{lemma}
Let $\mathcal{A}$ be a unistructural cluster algebra. Then $f: \mathcal{A}\to \mathcal{A}$ is a cluster automorphism if and only if $f$ is an automorphism of the ambient field which restricts to a permutation of the set of cluster variables. 
\end{lemma}

\begin{proof} Again because of the remark following Conjecture \ref{conjecture1}, it suffices to show sufficiency.
Let $f:\mathcal A \to \mathcal A$ be an automorphism of the ambient field permuting the set of cluster variables. Let $\mathbf x= (x_1, \dots, x_n)$ and $(\mathbf x, Q)$ be a seed of the cluster algebra $\mathcal A$. Consider a cluster algebra $\mathcal A_1$ generated by the initial seed $(f(\mathbf x), Q)$ where the variable $f(x_i)$ is associated with the same point $i\in Q_0$ as the variable $x_i$ in the seed $(\mathbf x, Q)$. 

Since the two seeds have the same quiver $Q$, and $f$ is a homomorphism of the field, the variables obtained by successive mutations of the seed $(f(\mathbf x), Q)$ are exactly the images under $f$ of the cluster variables obtained by the corresponding mutations of the seed $(\mathbf x, Q)$. 

Therefore, since $f$ is an automorphism permuting cluster variables of $\mathcal A$, the sets of cluster variables of $\mathcal A$ and of $\mathcal A_1$ coincide and $\mathcal A=\mathcal A_1$ as $\mathbb Z$-algebras. Because $\mathcal A$ is unistructural, the decomposition of the set of cluster variables into clusters is unique. Thus $f(\mathbf x)$ is a cluster of $\mathcal A$. By Corollary 2.7 of \cite{ASS1}, $f$ is a cluster automorphism. 
\end{proof}

We end the paper with a discussion of  unistructurality. For cluster algebras of Dynkin type,  unistructurality is derived from Theorem \ref{theorem1}, which we just proved (in Lemma \ref{lemma_5.1}). 

\begin{theorem}
\label{thm_DynkinUni}
Let $\mathcal A$ be a cluster algebra of Dynkin type, then $\mathcal A$ is unistructural. 
\end{theorem}

\begin{proof}
We first observe that the rank $n$ of $\mathcal A$ is determined by the set of cluster variables. This is the transcendence degree of the ambient field. The number of cluster variables is $n(n+3)/2$ in the case $\mathbb A_n$, it is $n^2$ in the case $\mathbb D_n$ and is equal to $42, 70$ or $128$ in the cases $\mathbb E_6$, $\mathbb E_7$ or $\mathbb E_8$, respectively. Therefore the number of cluster variables determines the type. 

Denote now by $\mathcal X$ the set of all cluster variables of $\mathcal A$ and suppose that $\mathcal A_1$ is another cluster algebra  of Dynkin type which has the same set $\mathcal X$ of cluster variables.  Then $\mathcal A$ and $\mathcal A_1$ are of the same type.  
Therefore, we can choose two seeds of these cluster algebras that share the same quiver. 

Let $(\mathbf x, Q)$ and $(\mathbf y, Q)$ be seeds of $\mathcal A$ and $\mathcal A_1$, respectively.  Define $f:\mathcal X\to \mathcal X$ by $f(x_i) = y_i$, where $x_i$ and $y_i$ are cluster variables of $\mathbf x$ and $\mathbf y$, respectively, corresponding to the same point $i\in Q_0.$

Because $\mathbf x$ and $\mathbf y$ are transcendence bases of the ambient field, $f$ extends to a field automorphism. 

Note that the variables obtained by successive mutations of $(\mathbf y, Q)$ are precisely the images under $f$ of the variables obtained by the corresponding mutations of $(\mathbf x, Q)$. Thus we obtain that $\mathcal X = f(\mathcal X)$ and therefore $f$ is an automorphism of the ambient field permuting the elements of $\mathcal X$. By Theorem \ref{theorem1} for algebras of Dynkin type, $f$ is a cluster automorphism. 

Thus  the image $\mathbf y$ under $f$ of the cluster $\mathbf x$, is a cluster in $\mathcal{A}$, and therefore $\mathcal A=\mathcal A_1$ as cluster algebras, that is their sets of clusters coincide. 
\end{proof}

Next we show Conjecture \ref{conjecture2} for cluster algebras of rank $2$. 

\begin{theorem}
Let $\mathcal A$ be a cluster algebra of rank $2$. Then $\mathcal A$ is unistructural. 
\end{theorem}

\begin{proof}
Let us fix a cluster structure on $\mathcal A$. Let $((x_1,x_2), Q)$ be the initial seed, where $Q$ denotes the quiver $Q= 1 \overset{r}{\rightarrow}2$ with $r\geq 0$ the number of arrows from $1$ to $2$. If $r=0$ or $r=1$, the algebra is of finite type and the result follows from Theorem \ref{thm_DynkinUni}. Let $r\geq 2$. Using the expansion formula of \cite{LS1}, we see that the only cluster variables that do not have both $x_1$ and $x_2$ in the denominator are $x_1, x_2, \frac{x_1^r+1}{x_2}, \frac{x_2^r+1}{x_1}$. Moreover there are exactly two clusters containing $x_1$, namely $(x_1, x_2)$ and $(x_1,\frac{x_1^r+1}{x_2})$. Suppose there is another cluster structure of rank 2 on $\mathcal A$ with the same set of variables. Then there exists a cluster $(x_1, y)$ with 
\begin{equation*}
y=\mathcal L(y, (x_1,x_2)) = \frac{P(x_1,x_2)}{x_1^{d_1} x_2^{d_2}}
\end{equation*}
such that $d_1>0$ and $P$ is a polynomial in $\mathbb Z[x_1, x_2]$ which is not divisible by $x_1$. Mutating this cluster in direction $y$ produces the following cluster
\begin{equation*}
(x_1, \frac{x_1^m+1}{P(x_1,x_2)} x_1^{d_1}x_2^{d_2})
\end{equation*}
where $m$ is the number of arrows in the quiver of the cluster $(x_1,y)$ in the second cluster structure. The Laurent phenomenon implies that 
\begin{equation*}
P(x_1,x_2) = Q(x_1,x_2) M(x_1,x_2),
\end{equation*}
where $M$ is a monomial and $Q$ is a divisor of $x^m_1+1$ in $\mathbb Z[x_1,x_2]$. Moreover, $x_1$ does not divide $M$ because it does not divide $P$. We consider two cases. If $M=1$ then $P=Q$ and so our cluster is
\begin{equation*}
(x_1, \frac{x_1^m+1}{P(x_1,x_2)} x_1^{d_1}x_2^{d_2}) = (x_1, P'(x_1,x_2))
\end{equation*}
where $P'$ is a polynomial in $\mathbb Z[x_1,x_2]$. Because of our description of the cluster variables, it follows that $P' = x_2$. If $M\neq 1$ then, because $x_1$ does not divide $M$, $M=x_2^{e_2}$ with $e_2\geq 1$. Thus the denominator of the cluster variable 
$\frac{x_1^m+1}{P(x_1,x_2)} x_1^{d_1}x_2^{d_2}$ is equal to $x_2^{-d_2+e_2}$. Again because of our description of the cluster variables, this cluster variable is equal to $\frac{x_1^r+1}{x_2}$. This implies that $d_1=0$, a contradiction. 
\end{proof}

{\bf Acknowledgements.} The first and third authors gratefully acknowledge partial support from NSERC of Canada, FQRNT and the University of Sherbrooke. 
The second author was  supported by NSF grants 1001637 and 1254567 and by the University of Connecticut.


\begin{thebibliography}{}
%\bibitem[{ABS1}]{ABS}{I. Assem,
%  T. Br\"ustle and R. Schiffler}, Cluster-tilted 
%  algebras  as trivial extensions,  \emph{Bull. Lond. Math. Soc.\/} {\bf 40} (2008), 151--162.
\bibitem[{ABS2}]{ABS2}{I. Assem,
  T. Br\"ustle and R. Schiffler}, Cluster-tilted algebras and slices, \emph{J. of Algebra\/} {\bf 319} (2008), 3464-3479. 
%\bibitem[{ABS3}]{ABS3}{I. Assem,
%  T. Br\"ustle and R. Schiffler}, On the Galois covering of a cluster-tilted algebra, \emph{J. Pure Appl. Alg.\/} {\bf 213} (7) (2009) 1450--1463.
\bibitem[ABS4]{ABS4}{I. Assem, T. Br\"ustle and R. Schiffler}, Cluster-tilted algebras without clusters, \emph{J. Algebra\/} {\bf 324}, (2010) 2475-2502.
\bibitem[{AD}]{AD}  I. Assem and G. Dupont, Modules over cluster-tilted algebras determined by their dimension vectors. to appear in {\em Comm. Algebra\/}. % arXiv:1202.5698.

%\bibitem[ADSS]{ADSS}  I. Assem, G. Dupont, R. Schiffler and D. Smith, Friezes, strings and
%cluster variables, to appear in \emph{Glasgow J. Math.}

\bibitem[{ASS}]{ASS} I. Assem, D. Simson and A. Skowro\'nski, {\em Elements of
  the Representation Theory of Associative Algebras, 1:
  Techniques of Representation Theory\/}, London Mathematical Society
Student Texts 65, Cambridge University Press, 2006.

\bibitem [{ASS1}]{ASS1} I. Assem, R. Schiffler, V. Shramchenko, Cluster automorphisms, {\it  Proc. London Math. Soc.}
  {\bf 3} no. 104, 1271-1302 (2012).

 
\bibitem[BM]{BM} { A. B. Buan, and R. Marsh},
Denominators in cluster algebras of affine type, {\it J. Algebra\/}, {\bf 323}, 8 (2010), 2083-2102.

\bibitem[{BMRRT}]{BMRRT}  { A. B. Buan, R. Marsh, M. Reineke, I. Reiten and
  G. Todorov}, Tilting theory and cluster combinatorics,
  \emph{Adv. Math.} {\bf 204}   (2006),  no. 2, 572-618.
  
\bibitem[BMR]{BMRdenominators} { A. B. Buan, R. Marsh, and I. Reiten }
Denominators of cluster variables {\it J. Lond. Math. Soc.\/} (2) {\bf 79} (2009), no. 3, 589-611.

\bibitem[BMR1]{BMR1} A. B. Buan, R. Marsh and I. Reiten, Cluster-tilted algebras, \emph{Trans. Amer. Math. Soc.} {\bf 359}, (2007), no. 1, 323-332.

%\bibitem[{CCS}]{CCS}  { P. Caldero, F. Chapoton and
%R. Schiffler}, Quivers with relations arising from clusters ($A_n$
%case), \emph{Trans. Amer. Math. Soc.} {\bf 358} (2006), no. 3, 1347-1364. 
%
%\bibitem[CK]{CK06} { P. Caldero and B. Keller}, From triangulated categories to
%  cluster algebras, \emph{Invent. Math.} {\bf 172}  (2008), no. 1, 169-211.
%  
\bibitem[CK2]{CK2} P. Caldero and B. Keller. From triangulated categories to cluster
algebras II. \emph{Ann. Sc. Ec. Norm. Sup.}
{\bf 39} (2006), no. 4, 83-100. 


\bibitem[FST]{FST} S. Fomin, M. Shapiro, and D. Thurston, Cluster algebras and triangulated surfaces. Part I: Cluster complexes, \emph{Acta Mathematica} {\bf 201} (2008), 83-146.

%
%\bibitem[FT]{FT} S. Fomin and D. Thurston, Cluster algebras and triangulated surfaces. Part II: Lambda Lengths, preprint (2008),

  \bibitem[{FZ1}]{FZ1}  { S. Fomin and A. Zelevinsky},
 Cluster algebras I. Foundations, \emph{J. Amer. Math. Soc.}
{\bf 15} (2002), no. 2, 497-529 (electronic). 
\bibitem[FZ4]{FZ4} S. Fomin and A. Zelevinsky, Cluster algebras IV: Coefficients, \emph{Compositio Mathematica} \bf 143 \rm (2007), 112-164.

%\bibitem[{K}]{K}  { B. Keller}, On triangulated orbit categories,
%  \emph{Documenta Math.} {\bf 10} (2005), 551-581.
%  

\bibitem[KQ]{KQ} 
Y Kimura and F. Qin,
Graded quiver varieties, quantum cluster algebras, and dual canonical basis, arXiv:1205.2066v2.

\bibitem[LS]{LS4} 
K. Lee and R. Schiffler, Positivity for cluster algebras,
 arXiv:1306.2415.

\bibitem[LS1]{LS1} K. Lee and R. Schiffler, A combinatorial algebra for rank 2 cluster variables, \emph{J. Alg. Comb.} {\bf 37}, (2013), no. 1, 67-85.

\bibitem [MSW] {MSW} G. Musiker, R. Schiffler and L. Williams, Positivity for cluster algebras from surfaces, \emph{Compositio Mathematica} {\bf 149} \rm (2013), no. 2, 217-263.

\bibitem[P]{Palu} Y. Palu, Cluster characters II: A multiplication formula, \emph{Proc. London Math. Soc.}, {\bf 104} (2012), no. 1, 57-78. 

\bibitem[S]{S} I. Saleh, On the automorphisms of cluster algebras, arXiv: 1011.0894.


\end{thebibliography}
\end{document}